\documentclass[12pt]{amsart}
\usepackage[osf,sc]{mathpazo}
\usepackage{amssymb}

\usepackage{geometry}
\usepackage{bm}
\usepackage{graphicx}
\usepackage{tikz-cd}
\usepackage{hyperref}
\hypersetup{
    colorlinks=true, %set true if you want colored links
    linktoc=all,     %set to all if you want both sections and subsections linked
    linkcolor=blue,
        citecolor=red,
    filecolor=black,
    urlcolor=blue	%choose some color if you want links to stand out
}
\usepackage{enumerate}
\usepackage[inline]{enumitem}
\makeatletter
\newcommand{\inlineitem}[1][]{%
\ifnum\enit@type=\tw@
    {\descriptionlabel{#1}}
  \hspace{\labelsep}%
\else
  \ifnum\enit@type=\z@
       \refstepcounter{\@listctr}\fi
    \quad\@itemlabel\hspace{\labelsep}%
\fi} \makeatother
\newcommand{\ga}{\alpha}
\newcommand{\gb}{\beta}
\newcommand{\gga}{\gamma}
\newcommand{\gd}{\delta}

\newcommand{\gl}{\lambda}

\newcommand{\gp}{\pi}

\newcommand{\gs}{\sigma}

\newcommand{\Gs}{\Sigma}

\newcommand{\Gom}{\Omega}

\newcommand{\sbnq}{\subsetneq}

\newcommand{\bs}{\backslash}

\newcommand{\nin}{\notin}

\newcommand{\mbb}{\mathbb}

\newcommand{\mcl}{\mathcal}

\newcommand{\ol}{\overline}

\newcommand{\us}{\underset}
\newcommand{\os}{\overset}

\newcommand{\lla}{\longleftarrow}
\newcommand{\lra}{\longrightarrow}

\newcommand{\N}{\mbb N}
\newcommand{\Z}{\mbb Z}

\newcommand{\ra}{\rightarrow}
\newcommand{\Ra}{\Rightarrow}

\newcommand{\es}{\emptyset}

\newcommand{\equ}[1]{%
\begin{equation*}
#1
\end{equation*}
}
\newcommand{\equa}[1]{%
\begin{equation*}
\begin{aligned}
#1
\end{aligned}
\end{equation*}
}

\DeclareMathOperator{\Det}{Det}

\newcommand{\mattwo}[4]{%
\begin{pmatrix}
  #1 & #2\\ #3 & #4
\end{pmatrix}
}

\newcommand{\matcoltwo}[2]{%
\begin{pmatrix}
  #1\\#2
\end{pmatrix}
}

\theoremstyle{plain}
\newtheorem{theorem}{Theorem}[section]

\newtheorem{prop}[theorem]{Proposition}

\newtheorem{cor}[theorem]{Corollary}

\makeatletter
\def\namedlabel#1#2{\begingroup
   \def\@currentlabel{#2}%
   \label{#1}\endgroup
}
\makeatother
\newtheorem*{thmOmega}{\bf{Theorem} $\bm{\Gom}$}
\newtheorem*{thmSigma}{\bf{Theorem} $\bm{\Gs}$}
\theoremstyle{definition}
\newtheorem{defn}[theorem]{Definition}

\theoremstyle{remark}
\newtheorem{remark}[theorem]{Remark}
\newtheorem{example}[theorem]{Example}
\numberwithin{equation}{section}

\begin{document}
\title[On a Projective Space Invariant of a Co-torsion Module of Rank Two]{On a Projective Space Invariant of a Co-torsion Module of Rank Two over a Dedekind Domain}
\author{C P Anil Kumar}
\address{School of Mathematics, Harish-Chandra Research Institute, HBNI, Chhatnag Road, Jhunsi, Prayagraj (Allahabad), 211 019,  India. \,\, email: {\tt akcp1728@gmail.com}}
\subjclass[2010]{13F05, 11R42}
\keywords{Dedekind Domains, Modules over Dedekind Domains, Projective Spaces Associated to Ideals, Zeta Functions}
\thanks{This work is done while the author is a Post Doctoral Fellow at Harish-Chandra Research Institute, Prayagraj(Allahabad).}
\date{\sc \today}
\begin{abstract}
For a Dedekind domain $\mcl{O}$ and a rank two co-torsion module $M\subseteq \mcl{O}^2$ with invariant factor ideals $\mcl{L}\supseteq \mcl{K}$ in $\mcl{O}$,  that is, $\frac{\mcl{O}^2}{M}\cong \frac{\mcl{O}}{\mcl{L}}\oplus \frac{\mcl{O}}{\mcl{K}}$, we associate a new projective space invariant element in $\mbb{PF}^1_{\mcl{I}}$ where $\mcl{I}$ is given by the ideal factorization $\mcl{K}=\mcl{L}\mcl{I}$ in $\mcl{O}$. This invariant element along with the invariant factor ideals determine the module $M$ completely as a subset of $\mcl{O}^2$. As a consequence, projective spaces associated to ideals in $\mcl{O}$ can be used to enumerate such modules. We compute the zeta function associated to such modules in terms of the zeta function of the one dimensional projective spaces for the ring $\mcl{O}_K$ of integers in a number field $K/\mbb{Q}$ and relate them to Dedekind zeta function. Using the projective spaces as parameter spaces, we re-interpret the Chinese remainder reduction isomorphism $\mbb{PF}^1_{\mcl{I}} \ra \us{i=1}{\os{l}{\prod}}\mbb{PF}^1_{\mcl{I}_i}$ associated to a factorization of an ideal $\mcl{I}=\us{i=1}{\os{l}{\prod}}\mcl{I}_i$ into mutually co-maximal ideals $\mcl{I}_i,1\leq i\leq l$ in terms of the intersection of associated modules arising from the projective space elements.
\end{abstract}
\maketitle
\section{\bf{Introduction}}
Projective spaces associated to ideals in commutative ring with unity are of immense interest in number theory and geometry. We explore here one such interest. 
For two fixed ideals $\mcl{L}\supseteq \mcl{K}$ in a Dedekind domain $\mcl{O}$, the one dimensional projective space associated to the ideal $\mcl{I}$ where $\mcl{K}=\mcl{L}\mcl{I}$ forms a parameterizing space for co-torsion modules (of rank two) in $\mcl{O}^2$. This is the content of the two main Theorems~[\ref{theorem:ProjectiveInvariant},\ref{theorem:ProjectiveInvariantSurj}] of this article. As a consequence, they can be used for enumeration purposes if the projective spaces are finite sets (refer to Theorems~[\ref{theorem:Bijection},\ref{theorem:ZetaFunction}]). In Theorem~\ref{theorem:ZetaFunction} we relate the three zeta functions $\zeta_{\mcl{O}_K},\zeta_{\mcl{O}_K^2},\zeta^{\mcl{O}_K}_{\mbb{PF}^1}$ as $\zeta_{\mcl{O}_K^2}(s)=\zeta_{\mcl{O}_K}(2s)\zeta^{\mcl{O}_K}_{\mbb{PF}^1}(s)$. As an example, the cardinality of finite index subgroups of $\Z^2$ (refer to Theorem~\ref{theorem:NumberofSubgroupsofFiniteIndex}) and its zeta functions can be calculated (refer to Corollary~\ref{cor:ZetaFunctions}) in terms of the zeta function of the one dimensional projective spaces over integers. By exploring the properties of these parameterizing spaces $\mbb{PF}^1_{\mcl{I}},\mcl{I}\subseteq \mcl{O}$, we can also study the corresponding modules. For example we can re-interpret the Chinese remainder reduction map associated to a finite product of mutually co-maximal ideals, in terms of the intersection of such modules as given in Theorem~\ref{theorem:IntersectionModules}. For the theory of Dedekind domains and modules over Dedekind domains refer to N.~Bourbaki~\cite{MR1727221}, chapter VII, sections $2 \& 4$. 
\section{\bf{Statement of the Main Theorems}}
In this section we give the required definitions in order to state the main theorems of the article.
\begin{defn}
Let $\mcl{O}$ be a domain. A torsion free $\mcl{O}$-module $M$ is said to be of rank $n\in \N$, if the cardinality of a maximal set in $M$ consisting $\mcl{O}$-linearly independent elements is $n$. Note that such a set always gives rise to a basis for the $S^{-1}\mcl{O}$-vector space $S^{-1}M$ where $S=\mcl{O}\bs \{0\}$.
\end{defn}
\begin{example}
With this definition the module $\mbb{Q}$ is a rank one $\Z$-module. However $\mbb{Q}$ is not a finitely generated $\Z$-module.
\end{example}
\begin{defn}[Co-torsion Module]
Let $\mcl{O}$ be a Dedekind domain. An $\mcl{O}$-submodule $M\subseteq \mcl{O}^n$ is said to be co-torsion if $\frac{\mcl{O}^n}{M}$ is torsion. Note that the ambient module $\mcl{O}^n$ is also important in this definition. In   N.~Bourbaki~\cite{MR1727221}, according to Definition 1, Page 512, chapter VII, such a module $M$ becomes a lattice of $(S^{-1}\mcl{O})^n$ with respect to $\mcl{O}$ where $S=\mcl{O}\bs \{0\}$.
\end{defn}
\begin{remark}
For $n=2$, a co-torsion $\mcl{O}$-module $M\subseteq \mcl{O}^2$ must necessarily have rank two. Conversely a rank two torsion-free $\mcl{O}$-module $M \subseteq \mcl{O}^2$ must be a co-torsion module. However a rank two torsion-free $\mcl{O}$-module need not be embeddable in $\mcl{O}^2$. An example is the rank two $\Z$-module $\mbb{Q}\oplus \mbb{Q}$ which is not embeddable in $\Z \oplus \Z$. A finitely generated rank two torsion-free $\mcl{O}$-module $M$ is embeddable in $\mcl{O}^2$. In fact such a module $M\cong \mcl{O}\oplus \mcl{I}$ for some non-zero ideal $\mcl{I}\subseteq \mcl{O}$.
\end{remark}
%\begin{defn}[Height ideal of a Co-torsion Module]
%Let $\mcl{O}$ be a Dedekind domain and let $\mcl{O}$-submodule $M\subseteq \mcl{O}^2$ be a co-torsion module. %The height ideal of $M$ is defined as \equ{\mcl{H}_M=\bigcap\big\{\mcl{I}\mid \mcl{I}\subs \mcl{O}\text{ is an %ideal and }\mcl{I}\mcl{O}^2=\mcl{I}\oplus \mcl{I}\supseteq M\big\}.} 	
%It is an ideal in $\mcl{O}$ and $\mcl{H}_M\oplus \mcl{H}_M \supseteq M$.
%\end{defn}
\begin{defn}[Projective Space Associated to an Ideal]
Let $\mcl{O}$ be a Dedekind domain and $\mcl{I}\subseteq \mcl{O}$ be an ideal. Let $(\mcl{O}^2)^{*}=\{(a,b)\in \mcl{O}^2\mid \langle a\rangle +\langle b\rangle = \mcl{O}\}$. Define an equivalence relation $(\mcl{O}^2)^{*}$ as follows. We say $(a,b)\sim (c,d)$ if $ad-bc\in \mcl{I}$. We define the one dimensional projective space $\mbb{PF}^1_{\mcl{I}}$ associated to the ideal $\mcl{I}$ to be the set of equivalence classes namely $\mbb{PF}^2_{\mcl{I}}=\frac{(\mcl{O}^2)^{*}}{\sim}$. The equivalence of the element $(a,b)\in (\mcl{O}^2)^{*}$ is denoted by $[a:b]$.
\end{defn}
\begin{remark}
Let $\mcl{O}$ be a Dedekind domain and $(0)\neq \mcl{I}\subseteq \mcl{O}$ be a non-zero ideal. Then $(a,b)\sim (c,d)$ if and only if there exists an element $\gl\in \mcl{O}$ such that $\ol{\gl} \in \mcl{U}\big(\frac{\mcl{O}}{\mcl{I}}\big)$ the unit group of the ring $\frac{\mcl{O}}{\mcl{I}}$ and $c=\gl a,d=\gl b$.	
\end{remark}
\begin{defn}[Invariants Associated to a Co-torsion Module]
\label{defn:Invariants}
Let $\mcl{O}$ be a dede-kind domain and let $\mcl{O}$-submodule $M\subsetneq \mcl{O}^2$ be a co-torsion module. 
We associate the following invariants for the co-torsion module $M$.
\begin{enumerate}
\item Elementary Maximal Ideal Divisor Invariants: There exists finitely many maximal ideals $\mcl{M}_i\sbnq \mcl{O}$ and integers $k_i,l_i,1\leq i\leq r$ such that $0\leq l_i\leq k_i\neq 0$ and $\frac{\mcl{O}^2}{M}\cong \us{i=1}{\os{r}{\bigoplus}}\bigg(\frac{\mcl{O}}{\mcl{M}_i^{l_i}}\oplus \frac{\mcl{O}}{\mcl{M}_i^{k_i}}\bigg)$. 
\item Invariant Factor Ideals: These ideals are defined as $\mcl{L}=\mcl{M}_1^{l_1}\mcl{M}_2^{l_2}\ldots \mcl{M}_r^{l_r}$, $\mcl{K}=\mcl{M}_1^{k_1}\mcl{M}_2^{k_2}\ldots \mcl{M}_r^{k_r}$ with $\mcl{L}\supseteq \mcl{K}$ and $\frac{\mcl{O}^2}{M}\cong \frac{\mcl{O}}{\mcl{L}}\oplus \frac{\mcl{O}}{\mcl{K}}$.
\item Projective Space Invariant: For the $\mcl{O}$-module $M$ define another invariant ideal $\mcl{I}=\mcl{M}_1^{k_1-l_1}\mcl{M}_2^{k_2-l_2}\ldots$ $ \mcl{M}_r^{k_r-l_r}$ with the convention that $\mcl{I}=\mcl{O}$ if $k_i=l_i$ for all $1\leq i\leq r$.	The projective space invariant of the module $M$ is the space $\mbb{PF}^1_{\mcl{I}}$ where $\mcl{I}$ is the projective space invariant ideal of $M$. Here $\mcl{K}=\mcl{L}\mcl{I}$ as a product of ideals in $\mcl{O}$.
\item Projective Space Invariant Element: We will show later in Proposition~\ref{prop:ProjSpaceElement} that $\mcl{L}\oplus \mcl{L} \supseteq M$. Now we associate for $(ta,tb)\in M$ with $t\in \mcl{L}\bs \big(\us{i=1}{\os{r}{\cup}}\mcl{L}\mcl{M}_i\big),a,b\in \mcl{O}$ and $\langle a\rangle +\langle b\rangle =\mcl{O}$, the invariant element $[a:b]\in \mbb{PF}^1_{\mcl{I}}$. We will show later again in Proposition~\ref{prop:ProjSpaceElement} that there exists such an element $(ta,tb)\in M$ and the element $[a:b]\in \mbb{PF}^1_{\mcl{I}}$ is uniquely determined.
\end{enumerate}
\end{defn}
Now we state the first and second main theorems of the article. 
\begin{thmOmega}
\namedlabel{theorem:ProjectiveInvariant}{$\Gom$}
Let $\mcl{O}$ be a Dedekind domain and let $\mcl{O}$-submodule $M\subsetneq \mcl{O}^2$ be a co-torsion module. With the notations as in Definition~\ref{defn:Invariants} the invariant ideals $\mcl{L}\supseteq \mcl{K}$ and the invariant element $[a:b]\in  \mbb{PF}^1_{\mcl{I}}$ where $\mcl{K}=\mcl{L}\mcl{I}$ completely determine the $\mcl{O}$-submodule $M$ as a subset of $\mcl{O}^2$. Stating in other words, if $M_i\subsetneq \mcl{O}^2$ be two such modules with invariants $\mcl{L}_i\supseteq \mcl{K}_i$ and $[a_i,b_i]\in \mbb{PF}^1_{\mcl{I}_i}$ where $\mcl{K}_i=\mcl{L}_i\mcl{I}_i,i=1,2$ then we have 
\equa{M_1&=M_2 \text{ if and only if }\\ \mcl{L}_1&=\mcl{L}_2,\mcl{K}_1=\mcl{K}_2,[a_1:b_1]=[a_2:b_2]\in \mbb{PF}^1_{\mcl{J}} \text{ where }\mcl{J}=\mcl{I}_1=\mcl{I}_2.}
\end{thmOmega}

\begin{thmSigma}
\namedlabel{theorem:ProjectiveInvariantSurj}{$\Gs$}
Let $\mcl{O}$ be a Dedekind domain. Let $\mcl{L}\supseteq \mcl{K}\sbnq \mcl{O}$ be two ideals and $[a:b]\in \mbb{PF}^1_{\mcl{I}}$ where $\mcl{I}$ is given by the ideal factorization $\mcl{K}=\mcl{L}\mcl{I}$. Then there exists an $\mcl{O}$-submodule $M\subsetneq \mcl{O}^2$ which has the invariant ideals $\mcl{L}\supseteq \mcl{K}$ and has the projective space invariant element $[a:b]\in \mbb{PF}^1_{\mcl{I}}$.
\end{thmSigma}
\section{\bf{Rank Two Co-torsion Modules over Principal Ideal Domains}}
In this section we consider rank two co-torsion modules $M\subsetneq \mcl{O}^2$ where $\mcl{O}$ is a principal ideal domain. We prove in this section the analogue of main Theorems~[\ref{theorem:ProjectiveInvariant},\ref{theorem:ProjectiveInvariantSurj}] for a principal ideal domains $\mcl{O}$ and give one application in Theorem~\ref{theorem:NumberofSubgroupsofFiniteIndex}. For an ideal $\mcl{I}=(d)$, $\mbb{PF}^1_{(d)}$ is denoted by just $\mbb{PF}^1_d$ as an abuse of notation even though $d$ is defined up to an associate. 
\begin{theorem}
\label{theorem:ProjInvariantPID}
Let $\mcl{O}$ be a principal ideal domain. Let $M\subsetneq \mcl{O}^2$ be a co-torsion module of rank two with elementary prime divisors $p_1,p_2,\ldots,p_r$ and integers $0\leq l_i\leq k_i\neq 0$ such that $\frac{\mcl{O}^2}{M}\cong \us{i=1}{\os{r}{\bigoplus}}\bigg(\frac{\mcl{O}}{p_i^{l_i}}\oplus \frac{\mcl{O}}{p_i^{k_i}}\bigg)$.
Let $d=p_1^{k_1-l_1}p_2^{k_2-l_2}\ldots p_r^{k_r-l_r}$. Then there is a unique projective space element 
$[a:b]\in \mbb{PF}^1_{d}$ such that $(d_1a,d_1b)\in M$ and for all elements $(d_1x,d_1y)\in M$ with $gcd(x,y)=1$ we have $[x:y]=[a:b]$ where $d_1=p_1^{l_1}p_2^{l_2}\ldots p_r^{l_r}$. Moreover if $d_2=p_1^{k_1}p_2^{k_2}\ldots p_r^{k_r}$ (has to be a non-unit) then for any matrix $\mattwo xyzw\in GL_2(\mcl{O})$ with $[x:y]=[a:b]$ the set $\{(d_1x,d_1y),(d_2z,d_2w)\}$ is a basis for $M$. This also yields a construction of a unique co-torsion module $M \sbnq \mcl{O}^2$ of rank two which has the invariant divisors $d_1,d_2$ and has the projective space invariant element $[a:b]\in \mbb{PF}^1_d$ where $d=\frac {d_2}{d_1}$. 
\end{theorem}
\begin{proof}
First of all we observe that the module $M\subsetneq \mcl{O}^2$ is a free module of rank two. Let $\{(s,t),(u,v)\}$ be a basis of $M$. Then by reducing the matrix $\mattwo stuv$ into smith normal form $\mattwo {d_1}00{d_2}$ with $d_1\mid d_2$, we obtain $\frac{\mcl{O}^2}{M}\cong \frac{\mcl{O}}{(d_1)}\oplus \frac{\mcl{O}}{(d_2)}$. Here $d_1=p_1^{l_1}p_2^{l_2}\ldots p_r^{l_r},d_2=p_1^{k_1}p_2^{k_2}\ldots p_r^{k_r}$ with $d=d_2/d_1=p_1^{k_1-l_1}p_2^{k_2-l_2}\ldots p_r^{k_r-l_r}$. So we can always assume that there is a basis of $M$ of the form $\{(d_1x,d_1y),(d_2z,d_2w)\}$ where $\mattwo xyzw\in GL_2(\mcl{O})$. In particular there exists an element $(d_1a,d_1b)\in M$ with $gcd(a,b)=1$. Consider another $\mcl{O}$-linearly independent set $\{(d_1x'$, $d_1y'),(d_2z',d_2w')\}$. Then this is a basis for $M$ if and only there exists a matrix $\mattwo {\ga}{\gb}{\gga}{\gd}\in GL_2(\mcl{O})$ such that we have \equ{\mattwo {\ga}{\gb}{\gga}{\gd}\mattwo {d_1x}{d_1y}{d_2z}{d_2w}=\mattwo {d_1x'}{d_1y'}{d_2z'}{d_2w'}}
which happens if and only if $\frac{d_2}{d_1}=d\mid (x'y-xy')$. This proves the theorem.
\end{proof}
As a consequence we have the following theorem.
\begin{theorem}
\label{theorem:NumberofSubgroupsofFiniteIndex}
The set of subgroups of $\Z^2$ of finite index $n$ is in bijection with the disjoint union of integer projective spaces \equ{\us{d\mid n, \frac nd=\square}{\bigsqcup}\mbb{PF}^1_d.}
Under this bijection we have the following.
\begin{enumerate}
\item $M\in \mbb{PF}^1_n$ if and only if there exists $(a,b)\in M$ with $gcd(a,b)=1$ if and only if $\frac {\Z^2}{M}$ is a finite cyclic group.
\item $M\in \mbb{PF}^1_d$ for $d\mid n, \frac nd=\square$ if and only if $\frac{\Z^2}{M}\cong \frac{\Z}{\sqrt{\frac nd}} \oplus \frac{\Z}{\sqrt{nd}}$. 
\item The number of subgroups of $\Z^2$ of finite index $n$ is $\gs(n)$, the sum of divisors of $n$.
\end{enumerate} 
\end{theorem}
\begin{proof}
If $n=1$ the proof is trivial. So assume $n>1$.
The consequences $(1),(2)$ are immediate and we will prove $(3)$. For this we observe that for a prime $p$ and $k\in \N$, \equ{\#(\mbb{PF}^1_{p^k})=p^{k-1}(p+1).} We also have that the chinese remainder reduction map gives a bijection of the sets $\mbb{PF}^1_m$ and $\us{i=1}{\os{r}{\prod}}\mbb{PF}^1_{m_i}$ where $m=m_1m_2\ldots m_r$
and $gcd(m_i,m_j)=1,1\leq i\neq j\leq r$ (refer to Theorem $1.6$ in C.P. Anil Kumar~\cite{MR3887364}).
Now $(3)$ follows.
\end{proof}
We give one more consequence of the previous theorem. Here we relate the zeta function of $\Z^2$ with the zeta function of the one dimensional projective spaces over integers. 
\begin{cor}
\label{cor:ZetaFunctions}	
For, $s\in \mbb{C},\operatorname{Re}(s)>2$, we have the zeta function $\zeta_{\Z^2}(s)$ associated to subgroups of $\Z^2$ of finite index defined as
\equ{\zeta_{\Z^2}(s)=\us{n\in \N}{\sum} \frac{\us{d\mid n, \frac nd =\square}{\sum}\mid \mbb{PF}^1_d\mid}{n^s}}
is given by  \equ{\zeta_{\Z^2}(s)=\zeta(s-1)\zeta(s) \text{ where }\zeta(s)=\us{n\in \N}{\sum}\frac 1{n^s}\text{ the usual zeta function}.}
The zeta function $\zeta_{\mbb{PF}^1}(s)$ of the one dimensional projective spaces over integers defined as 
\equ{\zeta_{\mbb{PF}^1}(s)=\us{d\in \N}{\sum} \frac{\mid \mbb{PF}^1_d\mid}{d^s}}
satisfies the equation \equ{\zeta_{\Z^2}(s)=\zeta(2s)\zeta_{\mbb{PF}^1}(s).}
\end{cor}
\begin{proof}
\equa{\us{n\geq 1}{\sum} \frac{\us{d\mid n, \frac nd =\square}{\sum}\mid \mbb{PF}^1_d\mid}{n^s}&=\us{n\geq 1}{\sum}\frac{\gs(n)}{n^s}\\
	&=\us{p \text{-prime}}{\prod}\bigg(\us{k\geq 0}{\sum}\frac{\gs(p^k)}{p^{ks}}\bigg)\\
	&=\us{p\text{-prime}}{\prod}\bigg(\frac{1}{1-\frac 1{p^{s-1}}}\bigg)\bigg(\frac{1}{1-\frac 1{p^s}}\bigg)=\zeta(s-1)\zeta(s).}
To prove the second equality we observe that 
\equa{\zeta(2s)\us{d\geq 1}{\sum} \frac{\mid \mbb{PF}^1_d\mid}{d^s}&=\us{d\geq 1,m\geq 1}{\sum}\frac{\mid \mbb{PF}^1_d\mid}{(dm^2)^s}\\
&=\us{n\geq 1}{\sum} \frac{\us{d\mid n, \frac nd =\square}{\sum}\mid \mbb{PF}^1_d\mid}{n^s}\\ &=\zeta(s-1)\zeta(s).}
\end{proof}
\section{\bf{Preliminaries}}
In this section we prove some preliminaries which are required to prove the main theorem of the article for Dedekind domains.
\begin{remark}
We will assume the following basic facts about a Dedekind domain $\mcl{O}$. 
\begin{itemize}
\item Let $\mcl{F}$ be a finite set of maximal ideals in $\mcl{O}$. Let $\mcl{I}\subseteq \mcl{O}$ be any non-zero ideal. Then the set $\mcl{I}\bs \bigg(\us{\mcl{M}\in \mcl{F}}{\bigcup}\mcl{I}\mcl{M}\bigg) \neq \es$.
\item Let $p,q\in \mcl{O}$ with ideal factorization as given by $(p)=\mcl{M}^{k_1}\mcl{M}^{k_2}\ldots \mcl{M}^{k_r}$, $(q)=\mcl{N}^{l_1}\mcl{N}^{l_2}\ldots \mcl{N}^{l_s}$. Then $(pq)=\mcl{M}^{k_1}\mcl{M}^{k_2}\ldots \mcl{M}^{k_r}\mcl{N}^{l_1}\mcl{N}^{l_2}\ldots \mcl{N}^{l_s}, p\nin \mcl{M}_i^{k_i+1},1\leq i\leq r$ and $pq\nin \mcl{M}_i^{k_i+l_j+1}$ if $\mcl{M}_i=\mcl{N}_j$.  
\end{itemize}

\end{remark}
\begin{prop}
	\label{prop:LocalIntersectionModule}
Let $\mcl{O}$ be a Dedekind domain. Let $M\sbnq \mcl{O}^2$ be a rank two co-torsion module. Let $\mcl{M}_i,1\leq i\leq r$ be the elementary maximal ideal divisor invariants of $M$ as in Definition~\ref{defn:Invariants}. Then in the vector space $S^{-1}\mcl{O}\oplus S^{-1}\mcl{O}$ with $S=\mcl{O}\bs \{0\}$, we have \equ{M=\us{i=1}{\os{r}{\bigcap}}M_{\mcl{M}_i}\us{\mcl{M}\neq \mcl{M}_i,1\leq i\leq r, \mcl{M}-maximal}{\bigcap}\mcl{O}^2_{\mcl{M}}}
where $M_{\mcl{M}_i}$ is the localization of $M$ at the maximal ideal $\mcl{M}_i\sbnq\mcl{O}$ for $1\leq i\leq r$.
\end{prop}
\begin{proof}
For the module $M$ we have \equ{M=\us{\mcl{M}\in MaxSpec(\mcl{O})}{\bigcap}M_{\mcl{M}}.}	
Since \equ{\frac{\mcl{O}^2}{M} \cong \us{i=1}{\os{r}{\bigoplus}}\bigg(\frac{\mcl{O}}{\mcl{M}_i^{l_i}}\oplus\frac{\mcl{O}}{\mcl{M}_i^{k_i}}\bigg)}
we have \equ{\frac{\mcl{O}^2_{\mcl{M}}}{M_{\mcl{M}}}= \bigg(\frac{\mcl{O}^2}{M}\bigg)_{\mcl{M}}=0\text{ if }\mcl{M}\neq \mcl{M}_i,1\leq i\leq r \Ra M_{\mcl{M}}=\mcl{O}^2_{\mcl{M}}.}
Hence the proposition follows.
\end{proof}
Now we prove two more useful propositions.
\begin{prop}
	\label{prop:ProjSpaceBijection}
	Let $\mcl{O}$ be a Dedekind domain and $\mcl{M}\sbnq\mcl{O}$ be a maximal ideal. Let $\mcl{M}_{\mcl{M}}$ denote the corresponding maximal ideal in $\mcl{O}_{\mcl{M}}$ and let $k\in \N$. Then the map
	\equ{\mbb{PF}^1_{\mcl{M}^k} \lra \mbb{PF}^1_{\mcl{M}^k_{\mcl{M}}}}
	given by \equ{[a:b] \lra [\frac a1:\frac b1]} is well defined and is a bijection.
\end{prop}
\begin{proof}
Any element in  $\mbb{PF}^1_{\mcl{M}^k_{\mcl{M}}}$ is either of the form $[1:\frac bt]$ where $b\in \mcl{O},t\in \mcl{O}\bs \mcl{M}$ or of the form $[\frac as: 1]$ where $a\in \mcl{O},s\in \mcl{O}\bs \mcl{M}$.
Consider $[1:\frac bt]$ w.l.o.g. Now there exists a $t'\in \mcl{O}\bs\mcl{M}$ such that $tt'-1\in \mcl{M}^k$. Hence we have $[1:\frac bt]=[1:bt']\in \mbb{PF}^1_{\mcl{M}^k_{\mcl{M}}}$. Now $[1:bt']\in \mbb{PF}^1_{\mcl{M}^k}$ maps onto the element $[1:\frac bt]\in \mbb{PF}^1_{\mcl{M}^k_{\mcl{M}}}$. This proves surjectivity. The map is clearly injective as well. Hence the proposition follows.
\end{proof}
\begin{prop}
\label{prop:ProjSpaceElement}
Let $\mcl{O}$ be a Dedekind domain. Let $M\sbnq \mcl{O}^2$ be a rank two co-torsion module. Let $\frac{\mcl{O}^2}{M}\cong \frac{\mcl{O}}{\mcl{L}}\oplus\frac{\mcl{O}}{\mcl{K}}$ with $\mcl{L}\supseteq \mcl{K}$
the invariant factor ideals as in Definition~\ref{defn:Invariants}. Then 
\begin{enumerate}
\item $\mcl{L}\oplus \mcl{L}\supseteq M$.
\item If $\mcl{M}_i,1\leq i\leq r$ are the elementary maximal ideal divisor invariants of $M$ as in Definition~\ref{defn:Invariants}, then exists $t\in \mcl{L}\bs \big(\us{i=1}{\os{r}{\cup}}\mcl{L}\mcl{M}_i\big),a,b\in \mcl{O}$ with $\langle a\rangle +\langle b\rangle =\mcl{O}$ such that $(ta,tb)\in M$.
\end{enumerate}
\end{prop}
\begin{proof}
We prove $(1)$ first. For a maximal $\mcl{M}\neq \mcl{M}_i,1\leq i\leq r$ we have $(\mcl{L}\oplus \mcl{L})_{\mcl{M}}=\mcl{L}_{\mcl{M}} \oplus \mcl{L}_{\mcl{M}} = \mcl{O}_\mcl{M} \oplus \mcl{O}_\mcl{M}$. Similarly we have $M_{\mcl{M}}=\mcl{O}^2_{\mcl{M}}$ since $\big(\frac{\mcl{O}^2}{M}\big)_{\mcl{M}}=0$.
For $\mcl{M}=\mcl{M}_i$ we use Theorem~\ref{theorem:ProjInvariantPID} because of the additional fact that $\mcl{O}_{\mcl{M}}$ is a discrete valuation ring and hence a principal ideal domain. For $\mcl{M}=\mcl{M}_i$, let $\mcl{M}_{\mcl{M}}=\langle p_i \rangle$. Then $(\mcl{L}\oplus \mcl{L})_{\mcl{M}}=\mcl{L}_{\mcl{M}} \oplus \mcl{L}_{\mcl{M}}=\langle p_i^{l_i}\rangle\oplus \langle p_i^{l_i}\rangle$. Using Theorem~\ref{theorem:ProjInvariantPID}	we have that there exists a basis of $M_{\mcl{M}}$ of the form $\{(p_i^{l_i}x,p_i^{l_i}y),(p_i^{k_i}z,p_i^{k_i}w)\}$ with $\mattwo xyzw\in GL_2(\mcl{O}_{\mcl{M}})$. Hence we obtain $(\mcl{L}\oplus \mcl{L})_{\mcl{M}} \supseteq M_{\mcl{M}}$. So we get 
\equ{\mcl{L}\oplus \mcl{L}=\us{\mcl{M}\in MaxSpec(\mcl{O})}{\bigcap} \big(\mcl{L}\oplus \mcl{L}\big)_{\mcl{M}}
\supseteq \us{\mcl{M}\in MaxSpec(\mcl{O})}{\bigcap} M_{\mcl{M}} = M.}
This proves $(1)$. 

Now we prove $(2)$. For $\mcl{M}=\mcl{M}_i$, choose $p_i\in \mcl{M}_i\bs \bigg(\us{j=1}{\os{r}{\bigcup}}\mcl{M}_i\mcl{M}_j\bigg)$. Then we have $(\mcl{M}_i)_{\mcl{M}_i}=\langle p_i\rangle$ a principal ideal in $\mcl{O}_{\mcl{M}_i}$. Now there exist $x_i,y_i\in \mcl{O}_{\mcl{M}_i}$ such that $\langle x_i \rangle +\langle y_i \rangle=\mcl{O}_{\mcl{M}_i}$ and $(p_i^{l_i}x_i,p_i^{l_i}y_i)\in M_{\mcl{M}_i}$. $[x_i:y_i]$ defines an element in $\mbb{PF}^1_{(\mcl{M}_i^{k_i-l_i})_{\mcl{M}_i}}$. By using Proposition~\ref{prop:ProjSpaceBijection} there exist $a_i,b_i\in \mcl{O}$ with $\langle a_i\rangle +\langle b_i\rangle =\mcl{O}$ and $[a_i:b_i]=[x_i:y_i]\in \mbb{PF}^1_{(\mcl{M}_i^{k_i-l_i})_{\mcl{M}_i}}$ and $[a_i:b_i]$ also defines an element in $\mbb{PF}^1_{\mcl{M}_i^{k_i-l_i}}$. Using the surjectivity of the following chinese remainder reduction map (Theorem $1.6$ in C.P. Anil Kumar~\cite{MR3887364})
\equ{\mbb{PF}^1_{\mcl{I}}\lra \us{i=1}{\os{r}{\prod}}\mbb{PF}^1_{\mcl{M}_i^{k_i-l_i}}}
where $\mcl{I}=\us{i=1}{\os{r}{\prod}}\mcl{M}_{i}^{k_i-l_i}$ the projective space invariant ideal of $M$ as in Definition~\ref{defn:Invariants}, we obtain that there exist $a,b\in \mcl{O}$ with $\langle a\rangle +\langle b\rangle=\mcl{O}$ and $[a:b]=[a_i:b_i]\in \mbb{PF}^1_{\mcl{M}_i^{k_i-l_i}} \Ra [a:b]=[x_i:y_i]\in \mbb{PF}^1_{(\mcl{M}_i^{k_i-l_i})_{\mcl{M}_i}}$. So using Theorem~\ref{theorem:ProjInvariantPID} we obtain that $p_i^{l_i}(a,b)\in M_{\mcl{M}_i}$ for $1\leq i\leq r$. Hence we have the element
\equ{\bigg(a\us{i=1}{\os{r}{\prod}}p_i^{l_i},b\us{i=1}{\os{r}{\prod}}p_i^{l_i}\bigg)\in \us{i=1}{\os{r}{\bigcap}} M_{\mcl{M}_i}.}
Hence by Proposition~\ref{prop:LocalIntersectionModule} it belongs to $M$. By the choice of $p_i$ we have 
$t=\us{i=1}{\os{r}{\prod}}p_i^{l_i} \in \mcl{L}\bs\big(\us{i=1}{\os{r}{\cup}}\mcl{L}\mcl{M}_i\big)$. This proves the proposition.
\end{proof}
\section{\bf{Proof of the Main Theorems and Consequences}}
We prove main Theorems~[\ref{theorem:ProjectiveInvariant},\ref{theorem:ProjectiveInvariantSurj}] of the article in this section.
First we prove Theorem~\ref{theorem:ProjectiveInvariant}.
\begin{proof}
Proposition~\ref{prop:ProjSpaceElement} define a projective space element $[a:b]\in \mbb{PF}^1_{\mcl{I}}$
where $(ta,tb)\in M$ for some $t\in \mcl{L}\bs\big(\us{i=1}{\os{r}{\cup}}\mcl{L}\mcl{M}_i\big)$.
If $(sc,sd)\in M$ for some $s\in \mcl{L}\bs\big(\us{i=1}{\os{r}{\cup}}\mcl{L}\mcl{M}_i\big)$ with $[c:d]\in \mbb{PF}^1_{\mcl{I}}$ then we get that $(sc,sd)\in M_{\mcl{M}_i},1\leq i\leq r$ such that in $\mcl{O}_{\mcl{M}_i}$ we have $\langle s \rangle=\langle t\rangle = \langle p_i^{l_i} \rangle$ and $[c:d]=[a:b]\in \mbb{PF}^1_{(\mcl{M}_i^{k_i-l_i})_{\mcl{M}_i}}$. Now using bijectivity in  Proposition~\ref{prop:ProjSpaceBijection} we have $[c:d]=[a:b]\in \mbb{PF}^1_{\mcl{M}_i^{k_i-l_i}},1\leq i\leq r$. Using bijectivity again of the chinese remainder reduction map we have $[c:d]=[a:b]\in \mbb{PF}^1_{\mcl{I}}$. So the element $[a:b]\in \mbb{PF}^1_{\mcl{I}}$ is uniquely determined for the module $M$.

Conversely if the invariant ideals $\mcl{L}\supseteq \mcl{K}$ are given for a module $M$ such that $\frac{\mcl{O}^2}{M}\cong \frac{\mcl{O}}{\mcl{L}}\oplus \frac{\mcl{O}}{\mcl{K}}$ and $\mcl{L}=\us{i=1}{\os{r}{\prod}}\mcl{M}_i^{l_i},
\mcl{K}=\us{i=1}{\os{r}{\prod}}\mcl{M}_i^{k_i},\mcl{I}=\us{i=1}{\os{r}{\prod}}\mcl{M}_i^{k_i-l_i}$ then for $a,b\in \mcl{O}$ with $\langle a \rangle + \langle b \rangle =\mcl{O}$ we have that the element 
$[a:b]\in \mbb{PF}^1_{\mcl{I}}$ determines elements $[a:b]\in \mbb{PF}^1_{\mcl{M}_i^{k_i-l_i}},1\leq i\leq r$ and hence elements in $\mbb{PF}^1_{(\mcl{M}_i^{k_i-l_i})_{\mcl{M}_i}},1 \leq i\leq r$ using Proposition~\ref{prop:ProjSpaceBijection}. Now using the fact that $\mcl{O}_{\mcl{M}_i},1\leq i\leq r$ are DVRs and hence PIDs and using Theorem~\ref{theorem:ProjInvariantPID}, we have that the module $M_{\mcl{M}_i}\subseteq \mcl{O}^2_{\mcl{M}_i}$ is uniquely determined by $\mcl{L}_{\mcl{M}_i}\supseteq \mcl{K}_{\mcl{M}_i}$ and $[a:b]\in \mbb{PF}^1_{(\mcl{M}_i^{k_i-l_i})_{\mcl{M}_i}},1 \leq i\leq r$. Now using Proposition~\ref{prop:LocalIntersectionModule}, the module $M$ is uniquely determined. For that uniquely determined module $M$, for any $t\in \mcl{L}\bs\big(\us{i=1}{\os{r}{\cup}}\mcl{L}\mcl{M}_i\big)$ we moreover have $(ta,tb)\in M$. This proves the first main theorem.
\end{proof}
Now we prove the second main Theorem~\ref{theorem:ProjectiveInvariantSurj} of the article.
\begin{proof}
Let $\mcl{L}=\mcl{M}_1^{l_1}\mcl{M}_2^{l_2}\ldots \mcl{M}_r^{l_r} \subseteq \mcl{O},\mcl{K}=\mcl{M}_1^{k_1}\mcl{M}_2^{k_2}\ldots \mcl{M}_r^{k_r}\sbnq \mcl{O}$ with $0\leq l_i\leq k_i\neq 0$. Then $\mcl{I}=\mcl{M}_1^{k_1-l_1}\mcl{M}_2^{k_2-l_2}\ldots \mcl{M}_r^{k_r-l_r}\subseteq \mcl{O}$.	Let $p_i\in \mcl{M}_i\bs \bigg(\us{j=1}{\os{r}{\bigcup}}\mcl{M}_i\mcl{M}_j\bigg),1\leq i\leq r$. Since $[a:b]\in \mbb{PF}^1_{\mcl{I}}$ with $\langle a \rangle +\langle b\rangle =\mcl{O}$, let $ax-by=1$ for some $x,y\in \mcl{O}$. Let $M_i\sbnq \mcl{O}^2_{\mcl{M}_i}$ be the unique co-torsion $\mcl{O}_{\mcl{M}_i}$-module of rank two with invariant ideals $\mcl{L}_{\mcl{M}_i}=(\mcl{M}_i^{l_i})_{\mcl{M}_i} \supseteq \mcl{K}_{\mcl{M}_i}=(\mcl{M}_i^{k_i})_{\mcl{M}_i}$ and projective space invariant element $[a:b]\in \mbb{PF}^1_{(\mcl{M}_i^{k_i-l_i})_{\mcl{M}_i}}$ using Theorem~\ref{theorem:ProjInvariantPID} for principal ideal domains. Then choose the module $M\subseteq S^{-1}\mcl{O}\oplus S^{-1}\mcl{O}$ where $S=\mcl{O}\bs\{0\}$ as given by 
\equ{M=M_1\cap M_2 \cap \ldots \cap M_r \us{\mcl{M}\neq \mcl{M}_i,1\leq i\leq r, \mcl{M}-\text{maximal}}{\bigcap}\mcl{O}^2_{\mcl{M}}.}
First we have $M\subseteq \mcl{O}^2$. This follows because $M\subseteq \us{\mcl{M}-\text{maximal}}{\bigcap}\mcl{O}^2_{\mcl{M}}=\mcl{O}^2$. Next we prove that $M_{\mcl{M}_i}=M_i,1\leq i\leq r$. Clearly $M\subseteq M_i$ and $M_i$ is an $\mcl{O}_{\mcl{M}_i}$ module. Hence $M_{\mcl{M}_i}\subseteq M_i,1\leq i\leq r$.
For any $1\leq j\leq r$ let $(u,v)\in M_j$. Because of the choice of $M_j$, this $\mcl{O}_{\mcl{M}_j}$-module has basis \equ{\{m_1=\big(a\us{i=1}{\os{r}{\prod}}p_i^{l_i},b\us{i=1}{\os{r}{\prod}}p_i^{l_i}\big),m_2=\big(y\us{i=1}{\os{r}{\prod}}p_i^{k_i},x\us{i=1}{\os{r}{\prod}}p_i^{k_i}\big)\}.}
Note that $p_i\nin \mcl{M}_j$ for $1\leq i\neq j\leq r$.
So there exists $\ga,\gb\in \mcl{O}_{\mcl{M}_j}$ such that $(u,v)=\ga m_1+\gb m_2$.
Let $\ga=\frac es,\gb =\frac fs$ where $e,f\in \mcl{O},s\in \mcl{O}\bs\mcl{M}_j$ then $s(u,v)=em_1+fm_2$. Now we observe that 
\equ{em_1+fm_2\in \mcl{O}^2\cap M_1\cap M_2\cap \ldots \cap M_j\cap \ldots \cap M_r \subseteq M.}
So $(u,v)\in M_{\mcl{M}_j} \Ra M_j\subseteq M_{\mcl{M}_j}$. This proves that $M_{\mcl{M}_j}=M_j,1\leq j\leq r$. 

Let $\mcl{M} \neq \mcl{M}_i,1\leq i \leq r$ be a maximal ideal. Let \equ{q_i\in \mcl{M}_i\bs \bigg(\us{j=1}{\os{r}{\bigcup}}\mcl{M}_i\mcl{M}_j\bigcup \mcl{M}_i\mcl{M}\bigg),1\leq i\leq r.}
Then the elements $n_1,n_2\in M$ where 
\equ{n_1=\big(a\us{i=1}{\os{r}{\prod}}q_i^{l_i},b\us{i=1}{\os{r}{\prod}}q_i^{l_i}\big),n_2=\big(y\us{i=1}{\os{r}{\prod}}q_i^{k_i},x\us{i=1}{\os{r}{\prod}}q_i^{k_i}\big).}
We also observe that 
\equ{\Det \mattwo{a\us{i=1}{\os{r}{\prod}}q_i^{l_i}}{b\us{i=1}{\os{r}{\prod}}q_i^{l_i}}{y\us{i=1}{\os{r}{\prod}}q_i^{k_i}}{x\us{i=1}{\os{r}{\prod}}q_i^{k_i}}=\us{i=1}{\os{r}{\prod}}q_i^{l_i+k_i}} which is a unit in $\mcl{O}_{\mcl{M}}$. Hence 
\equ{\matcoltwo{\lla n_1\lra}{\lla n_2\lra} = \mattwo{a\us{i=1}{\os{r}{\prod}}q_i^{l_i}}{b\us{i=1}{\os{r}{\prod}}q_i^{l_i}}{y\us{i=1}{\os{r}{\prod}}q_i^{k_i}}{x\us{i=1}{\os{r}{\prod}}q_i^{k_i}}\in GL_2(\mcl{O}_{\mcl{M}}).} 
This proves that $M_{\mcl{M}}=\mcl{O}^2_{\mcl{M}}$. 

So this module $M$ has exactly invariants $\mcl{L}\supseteq \mcl{K}$. Also the projective space invariant element is $[a:b]\in \mbb{PF}^1_{\mcl{I}}$ which can be checked by localization. Hence the second main theorem follows.
\end{proof}
As a consequence of main Theorems~[\ref{theorem:ProjectiveInvariant},\ref{theorem:ProjectiveInvariantSurj}] we have the following theorem which we state without proof.
\begin{theorem}[Bijection/Enumeration Theorem]
	\label{theorem:Bijection}
	Let $\mcl{O}$ be a Dedekind domain. Then there is a bijection of the set of co-torsion submodules (of rank two) in $\mcl{O}^2$ having fixed invariant factor ideals $\mcl{L}\supseteq \mcl{K}$ with the projective space $\mbb{PF}^1_{\mcl{I}}$ where we have the ideal factorization $\mcl{K}=\mcl{L}\mcl{I}$.
\end{theorem}
Now we relate the zeta functions.
\begin{theorem}
	\label{theorem:ZetaFunction}
Let $K$ be a finite extension of $\mbb{Q}$ and $\mcl{O}_{K}$ be the ring of integers. Let $N(\mcl{I})=\mid \frac{\mcl{O}_K}{\mcl{I}}\mid$ for an ideal $0\neq \mcl{I}\subseteq \mcl{O}_K$. Define the following zeta functions.
\begin{itemize}
	\item $\zeta_{\mcl{O}_K}(s)=\us{0\neq \mcl{I}\subseteq \mcl{O}_K,\mcl{I} \text{ an ideal}}{\sum}\frac{1}{N(\mcl{I})^s}$ the Dedekind zeta function.
	\item $\zeta_{\mcl{O}^2_K}(s)=\us{M\subseteq \mcl{O}^2_K,\frac{\mcl{O}^2_K}{M}\text{ is torsion}}{\sum}\frac{1}{\bigl\rvert \frac{\mcl{O}^2_K}{M}\bigr\lvert^s}$ the zeta function of co-torsion modules in $\mcl{O}_K^2$.
	\item $\zeta^{\mcl{O}_K}_{\mbb{PF}^1}(s)=\us{0\neq \mcl{I}\subseteq \mcl{O}_K,\mcl{I} \text{ an ideal}}{\sum}\frac{\mid \mbb{PF}^1_{\mcl{I}}\mid}{N(\mcl{I})^s}$ the zeta function of the one dimensional projective spaces associated to ideals in the ring of integers of a number field $K/\mbb{Q}$.
\end{itemize} 
Then we have for $s\in \mbb{C},\operatorname{Re}(s)>2$
\begin{enumerate}
\item $\zeta_{\mcl{O}^2_K}(s)=\zeta_{\mcl{O}_K}(s-1)\zeta_{\mcl{O}_K}(s)$.
\item $\zeta_{\mcl{O}^2_K}(s)=\zeta_{\mcl{O}_K}(2s)\zeta^{\mcl{O}_K}_{\mbb{PF}^1}(s)$.
\end{enumerate}
\end{theorem}
\begin{proof}
Here we use the fact for a number field $N(\mcl{I})$ is finite for $0\neq \mcl{I}\subseteq \mcl{O}_K$. We also use the fact that for any integer $n>0$ the set of ideals $0\neq \mcl{I}\subseteq \mcl{O}_K$ such that $N(\mcl{I})=n$ is a finite set and the Dedekind zeta function $\zeta_{\mcl{O}_K}(s)$ coverges for $s\in \mbb{C},\operatorname{Re}(s)>1$.

First we observe that for a maximal ideal $\mcl{M}\sbnq \mcl{O}_K,k\in \N,\gp\in \mcl{M}\bs \mcl{M}^2$ we have 
\equ{\mid \mbb{PF}^1_{\mcl{M}^k}\mid =N(\mcl{M})^k+N(\mcl{M})^{k-1}.}
This follows because \equa{\mbb{PF}^1_{\mcl{M}^k}&=\bigg\{[1:\gp^tu]\mid \ol{u}\in \mcl{U}(\frac{\mcl{O}_K}{\mcl{M}^{k-t}}),0\leq t\leq (k-1)\bigg\}\\ &\bigcup\bigg\{[\gp^tu:1]\mid \ol{u}\in \mcl{U}(\frac{\mcl{O}_K}{\mcl{M}^{k-t}}),0< t\leq (k-1)\bigg\}\\ &\bigcup\bigg\{[1:0],[0:1]\bigg\}.}
Since the Chinese remainder reduction map 
\equ{\mbb{PF}^1_{\mcl{I}}\lra \us{i=1}{\os{r}{\prod}}\mbb{PF}^1_{\mcl{M}_i^{k_i}}}
is bijective for $\mcl{I}=\us{i=1}{\os{r}{\prod}} \mcl{M}_i^{k_i}\subseteq \mcl{O}_K$ by Theorem $1.6$ in C.P. Anil Kumar~\cite{MR3887364}, we have 
\equ{\mid \mbb{PF}^1_{\mcl{I}}\mid = \us{i=1}{\os{r}{\prod}}\big(N(\mcl{M}_i)^{k_i}+N(\mcl{M}_i)^{k_i-1}\big).}
We prove $(2)$ first.	
\equa{\zeta_{\mcl{O}_K}(2s)\zeta^{\mcl{O}_K}_{\mbb{PF}^1}(s)&=\us{0\neq \mcl{I},\mcl{L}\subseteq \mcl{O}_K,\mcl{I},\mcl{L} \text{ two ideals}}{\sum}\frac{\mid \mbb{PF}^1_{\mcl{I}}\mid}{N(\mcl{I}\mcl{L}^2)^{s}}\\
&=\us{0\neq \mcl{K},\mcl{L}\subseteq \mcl{O}_K,\mcl{K}\subseteq \mcl{L} \text{ two ideals}}{\sum}\frac{\mid \mbb{PF}^1_{\mcl{I}}\mid}{N(\mcl{K})^s N(\mcl{L})^{s}} \text{ where }\mcl{K}=\mcl{L}\mcl{I}\\
&=\us{M\subseteq \mcl{O}^2_K,\frac{\mcl{O}^2_K}{M}\text{ is torsion}}{\sum}\frac{1}{\bigl\rvert \frac{\mcl{O}^2_K}{M}\bigr\lvert^s} \text{ because }\bigl\rvert \frac{\mcl{O}^2_K}{M}\bigr\lvert=N(\mcl{L})N(\mcl{K})\\
&=\zeta_{\mcl{O}^2_K}(s).}
Now we prove similarly $(1)$.
\equa{\zeta_{\mcl{O}^2_K}(s)&=\us{M\subseteq \mcl{O}^2_K,\frac{\mcl{O}^2_K}{M}\text{ is torsion}}{\sum}\frac{1}{\bigl\rvert \frac{\mcl{O}^2_K}{M}\bigr\lvert^s}\\
&=\us{0\neq \mcl{K},\mcl{L}\subseteq \mcl{O}_K,\mcl{K}\subseteq \mcl{L} \text{ two ideals}}{\sum}\frac{\mid \mbb{PF}^1_{\mcl{I}}\mid}{N(\mcl{K})^s N(\mcl{L})^{s}} \text{ where }\mcl{K}=\mcl{L}\mcl{I}\\
&=\us{\mcl{M}\in \text{MaxSpec}(\mcl{O}_K)}{\prod}\bigg(\us{k\geq 0}{\sum}\frac{\big(N(\mcl{M})^k+N(\mcl{M})^{k-1}+\ldots+1\big)}{N(\mcl{M})^{ks}}\bigg)\\
&\text{ by rearranging terms similar to the case of integers }\Z\\
&=\us{\mcl{M}\in \text{MaxSpec}(\mcl{O}_K)}{\prod}\bigg(\frac{1}{1-\frac 1{N(\mcl{M})^{s-1}}}\bigg)\bigg(\frac{1}{1-\frac 1{N(\mcl{M})^s}}\bigg)\\
&=\zeta_{\mcl{O}_K}(s-1)\zeta_{\mcl{O}_K}(s).}
This completes the proof.
\end{proof}
Now we state another consequence of the main theorem.
\begin{theorem}
	\label{theorem:IntersectionModules}
Let $\mcl{O}$ be a Dedekind domain. Let $\mcl{F}$ be a finite set of maximal ideals in $\mcl{O}$. Let $M_i,1\leq i\leq l$ be finitely many co-torsion submodules of rank two in $\mcl{O}^2$ such that their invariant factor ideals are $\mcl{L}_i \supseteq \mcl{K}_i$ which are co-maximal that is $\mcl{K}_i+\mcl{K}_j=\mcl{O}$ for $1\leq i\neq j\leq l$. Let $\mcl{I}_i,1\leq i\leq l$ be ideals in $\mcl{O}$ such that $\mcl{K}_i=\mcl{L}_i\mcl{I}_i,1\leq i\leq r$. Moreover assume that if a maximal ideal $\mcl{M} \supseteq \mcl{K}_i$ for some $1\leq i\leq r$ then $\mcl{M}\in \mcl{F}$. For $1\leq i\leq l$, let $[a_i:b_i]\in \mbb{PF}^1_{\mcl{I}_i}$ be the associated projective space invariant element of the module $M_i$. Let $a,b\in \mcl{O}$ with $\langle a \rangle +\langle b \rangle=\mcl{O}$ and  \equ{[a:b]\in \mbb{PF}^1_{\us{i=1}{\os{l}{\prod}}\mcl{I}_i}} be such that $[a:b]=[a_i:b_i]\in \mbb{PF}^1_{\mcl{I}_i},1\leq i\leq l$. The following assertions hold true.
\begin{enumerate}
\item The module $\us{i=1}{\os{r}{\bigcap}} M_i$ is a co-torsion module of rank two in $\mcl{O}^2$. 
\item The invariant factor ideals associated to $\us{i=1}{\os{r}{\bigcap}} M_i$ is $\us{i=1}{\os{l}{\prod}}\mcl{L}_i=\mcl{L}\supseteq \mcl{K}=\us{i=1}{\os{l}{\prod}}\mcl{K}_i$
\item The projective space invariant element associated to $\us{i=1}{\os{r}{\bigcap}} M_i$ is $[a:b]\in \mbb{PF}^1_{\us{i=1}{\os{l}{\prod}}\mcl{I}_i}$.
\item For any element $t\in \mcl{L}\bs\big(\us{\mcl{M}\in \mcl{F}}{\cup}\mcl{L}\mcl{M}\big)$ we have $(ta,tb)\in \us{i=1}{\os{r}{\bigcap}} M_i$.
\end{enumerate}
\end{theorem}
\begin{proof}
We note that $\frac{\mcl{O}^2}{\us{i=1}{\os{r}{\bigcap}} M_i}$ is annihilated by $\mcl{K}$ hence a torsion module. Now we immediately have a injection $\frac{\mcl{O}^2}{\us{i=1}{\os{r}{\bigcap}} M_i} \hookrightarrow \us{i=1}{\os{l}{\bigoplus}} \frac{\mcl{O}^2}{M_i}\cong \frac{\mcl{O}}{\mcl{L}}\oplus \frac{\mcl{O}}{\mcl{K}}$.
This injective map is an isomorphism which is proved by localization.
The rest of the proof is straight forward and the arguments are similar to the proof of main Theorem~\ref{theorem:ProjectiveInvariant}.	
\end{proof}


\begin{thebibliography}{1}
	
\bibitem{MR3887364}
C.~P.~Anil Kumar,
\newblock {O}n the {S}urjectivity of {C}ertain {M}aps,
\newblock {\em Journal of Ramanujan Mathematical Society}, Vol. {\bf 33}, No. {\bf 4}, Dec. 2018, pp. 335-378, \url{http://jrms.ramanujanmathsociety.org/previous-issues}, \url{https://arxiv.org/pdf/1608.03728.pdf}, MR3887364	

\bibitem{MR1727221}
N.~Bourbaki,
\newblock {E}lements of {M}athematics: {C}ommutative {A}lgebra, {C}hapter 1-7,
\newblock {Translated from the French. Reprint of the 1989 English translation. Elements of Mathematics (Berlin), Springer-Verlag, Berlin}, 1998. xxiv+625 pp. ISBN: 3-540-64239-0, MR1727221

\end{thebibliography}
\end{document}